\documentclass{amsart}
\usepackage{amsmath,amsfonts,amssymb,amscd}
\usepackage{mathrsfs}  
\usepackage{fullpage}
\usepackage{enumitem}

\usepackage{tikz,tikz-cd}
\usepackage{hyperref}
\usepackage{xcolor}

\numberwithin{equation}{section}

\theoremstyle{definition}
\newtheorem{theorem}{Theorem}[section]

\newtheorem{corollary}[theorem]{Corollary} 
\newtheorem{definition}[theorem]{Definition} 
\newtheorem{example}[theorem]{Example}
\newtheorem{lemma}[theorem]{Lemma}
\newtheorem{proposition}[theorem]{Proposition}

\newtheorem{thmintro}{Theorem}

\newtheorem*{theorem*}{Theorem}

\DeclareMathOperator\Aut{Aut}
\DeclareMathOperator\Ext{Ext}

\DeclareMathOperator\gr{gr}

\DeclareMathOperator\id{id}
\DeclareMathOperator\soc{soc}

\DeclareMathOperator\tot{tot}

\newcommand\inv{^{-1}}
\newcommand\iso{\cong}
\newcommand\kk{\Bbbk}

\newcommand\cH{\mathcal H}

\newcommand\NN{\mathbb N}
\newcommand\ZZ{\mathbb Z}

\newcommand\fs{\mathfrak s}
\newcommand\ft{\mathfrak t}

\begin{document}

\title{Quiver down-up algebras of type A}

\author[Gaddis]{Jason Gaddis}
\author[Keeler]{Dennis Keeler}
\address{Miami University, Department of Mathematics, Oxford, Ohio 45056} 
\email{gaddisj@miamioh.edu,keelerds@miamioh.edu}


\begin{abstract}
We present a generalization of down-up algebras, originally defined by Benkart and Roby. These quiver down-up algebras arise as quotients of the double of the extended Dynkin quiver of type A. Under a certain non-degeneracy condition, we show that quiver down-up algebras are noetherian piecewise domains, and that they are twisted Calabi--Yau. Finally, we consider the isomorphism problem for graded quiver down-up algebras.
\end{abstract}

\maketitle

Throughout, let $\kk$ be an algebraically closed field of characteristic zero. Down-up algebras were introduced by Benkart and Roby \cite{BRdu,BRduADD}. They have drawn significant interest in representation theory \cite{kulk}, invariant theory \cite{KKZ6}, and the study of Hopf algebras \cite{BWhopf}. In this paper, we realize a family of algebras which may reasonably be considered as a generalization of down-up algebras to quivers with many vertices. 
 
\begin{definition}\label{defn.qdu}
Let $Q$ be the following quiver: 
\begin{equation}\label{eq.dblA}
\begin{tikzcd}
    &         &   	& e_0 \arrow[llld, "u_0"] \arrow[rrrd, "d_{n-1}", shift left=2] 	&           &		&          \\
e_1 \arrow[rr, "u_1"] \arrow[rrru, "d_0", shift left=2] 	&        	&                   	e_2 \arrow[r, "u_2"]	\arrow[ll, "d_1", shift left=2]			 		& \cdots   \arrow[r, "u_{n-3}"]  \arrow[l, "d_2", shift left=2] &	e_{n-2} \arrow[rr, "u_{n-2}"]	\arrow[l, "d_{n-3}", shift left=2] &           	& e_{n-1} \arrow[lllu, "u_{n-1}"] \arrow[ll, "d_{n-2}", shift left=2]
\end{tikzcd}
\end{equation}
For $\alpha,\beta,\gamma \in \kk^n$, the corresponding \emph{quiver down-up algebra of type A} is defined as the quotient of $\kk Q$ by the relations
\begin{align*}
	d_{i-1}u_{i-1}u_i &= \alpha_i u_id_iu_i + \beta_i u_iu_{i+1}d_{i+1} +\gamma_i u_i \\
	d_id_{i-1}u_{i-1} &= \alpha_i d_iu_id_i + \beta_i u_{i+1}d_{i+1}d_i +\gamma_i d_i,
\end{align*}
for $i\in Q_0$. We denote this algebra by $\cH(\alpha,\beta,\gamma)$. If $\gamma_i=0$ for all $i \in Q_0$, then we write $\cH(\alpha,\beta,0)$.
\end{definition}

The classical down-up algebras occur when $n=1$, so there is a single vertex and two loops ($u=u_0$, $d=d_0$) with parameters $\alpha=\alpha_0$, $\beta=\beta_0$, and $\gamma=\gamma_0$.
These algebras are known to be noetherian domains if and only if $\beta\neq 0$. When $\beta \neq 0$ and $\gamma=0$, they are cubic Artin--Schelter regular algebras of global dimension three. We generalize these results to the quiver down-up algebras from Definition \ref{defn.qdu}. 

In Section \ref{sec.basis} we provide background on quivers and path algebras with relations, and establish a $\kk$-basis for the quiver down-up algebras. In Section \ref{sec.gwa}, we show that quiver down-up algebras appear as generalized Weyl algebras over the algebra $\oplus_{i \in Q_0} R[x_i,y_i]$. This shows that the $H(\alpha,\beta,\gamma)$ are noetherian and piecewise domains when $\beta_i\neq 0$ for all $i \in Q_0$. Section \ref{sec.skgrp} realizes certain quiver down-up algebras as skew group rings of an Artin--Schelter regular algebra and provides the last pieces of proving the following theorem.

\begin{thmintro}[Theorem \ref{thm.CY}]
Suppose $\beta_i \neq 0$ for all $i \in Q_0$. Then $\cH=\cH(\alpha,\beta,0)$ is twisted-graded Calabi--Yau of global dimension three with Nakayama automorphism $\mu^{\cH}$ given by $\mu^{\cH}(u_i)=-\beta_{i-1}\inv u_i$ and $\mu^{\cH}(d_i)=-\beta_{i-1}\inv d_i$ for all $i \in Q_0$.
\end{thmintro}

A direct consequence of this theorem is that $\cH(\alpha,\beta,\gamma)$ is twisted Calabi--Yau whenever $\beta_i\neq 0$ for all $i$. See Corollary \ref{cor.CY}. The following theorem, which is analogous to \cite[Theorem]{KMP} in the single-vertex case, is proved in Section \ref{sec.main}.

\begin{thmintro}[Theorem \ref{thm.main}]
\label{thm.introA}
Let $\alpha,\beta,\gamma \in \kk^n$ and let $\cH=\cH(\alpha,\beta,\gamma)$. The following are equivalent.
\begin{enumerate}
	\item $\beta_i \neq 0$ for all $i \in Q_0$,
	\item $\cH$ is right (or left) noetherian,
	\item $\kk[u_id_i,d_{i-1}u_{i-1}]$ is a polynomial ring in two variables for all $i \in Q_0$,
	\item $\cH$ is a piecewise domain.
\end{enumerate}
\end{thmintro}

Finally, in Section \ref{sec.iso} we consider the isomorphism problem for the graded quiver down-up algebras. This is related to recent work in \cite{GZiso} on isomorphisms of certain quantized preprojective algebras.

\begin{thmintro}[Theorem \ref{thm.iso}]
Fix $n \geq 3$. Let $\alpha,\alpha' \in \kk^n$ and $\beta,\beta' \in (\kk^\times)^n$. Then $\cH(\alpha,\beta,0) \iso \cH(\alpha',\beta',0)$ if and only if there exists $\lambda\in(\kk^\times)^n$ and $k \in Q_0$ such that one of the following hold:
\begin{enumerate}
	\item $\alpha_i' = \lambda_{i+k}\lambda_{i+k-1}\inv \alpha_{i+k}$,\quad
	$\beta_i' = \lambda_{i+k+1} \lambda_{i+k-1}\inv \beta_{i+k}$, or

	\item $\alpha_i' = -\lambda_{n-i-k}\inv\lambda_{n-i-k-1} \beta_{n-i-k-1}\inv\alpha_{n-i-k-1}$, \quad
	$\beta_i' = \lambda_{n-i-k}\inv\lambda_{n-i-k-2}\inv \beta_{n-i-k-1}\inv$.
\end{enumerate}
\end{thmintro}

In a companion piece \cite{GK2}, the authors show that certain graded quiver down-up algebras appear as normal extensions of preprojective algebras. This gives an alternate way to realize some of the results in this paper.

\subsection*{Acknowledgements}
Gaddis is partially supported by an AMS–Simons Research Enhancement Grant for PUI Faculty. The authors thank Dan Rogalski for helpful conversations

\section{Quivers and a basis for quiver down-up algebras}
\label{sec.basis}

A \emph{quiver} $Q$ is a directed graph. More precisely, it is a set $Q_0$ of vertices and a set $Q_1$ of edges along with maps $\fs,\ft: Q_1 \to Q_0$. For each $a \in Q_1$, $\fs(a)$ is the \emph{source} and $\ft(a)$ is the \emph{target}. A \emph{path} in $Q$ is a string $p=a_0a_1\hdots a_m$ with $\ft(a_i)=\fs(a_{i+1})$ for $i = 0,\hdots,m-1$. The maps $\fs,\ft$ then extend naturally to the set of paths with $\fs(p)=\fs(a_0)$ and $\ft(p)=\ft(a_m)$. Throughout, we assume that $Q_0$ and $Q_1$ are finite.

The set of paths in $Q$ form a $\kk$-linear basis of the \emph{path algebra} $\kk Q$. Multiplication in $\kk Q$ is given by concatenation. That is, for paths $p$ and $q$, $p \cdot q = pq$ if $\ft(p)=\fs(q)$ and $p \cdot q = 0$ otherwise. At each vertex $k \in Q_0$ there is a trivial path $e_k$ which has the properties $e_k p = p$ (resp. $pe_k=p$) if $\fs(p)=k$ (resp. $\ft(p)=k$) and $0$ otherwise. For paths on $\kk Q$ as in \eqref{eq.dblA}, the subscripts should be taken modulo $n$.

There is a filtration on $Q$ by path length, which extends naturally to $\kk Q$. An element $f \in \kk Q$ is \emph{homogeneous} if $f$ is the sum of elements of the same path length. An ideal $I$ in $\kk Q$ is \emph{homogeneous} if it is generated by homogeneous elements.

\begin{proposition}\label{prop.basis}
Let $\cH=\cH(\alpha,\beta,\gamma)$ with $\alpha,\beta,\gamma \in \kk^n$.
A $\kk$-basis for $\cH$ consists of paths of the form:
\begin{align}\label{eq.Aform}
    (u_iu_{i+1}\cdots u_{i+k})(d_{i+k}u_{i+k})^j(d_{i+k}d_{i+k-1} \cdots d_{i+k-\ell}),
\end{align}
with $i\in Q_0$ and $j,k,\ell \geq 0$. 
\end{proposition}
\begin{proof}
Let $Q$ be as in \eqref{eq.dblA}. We order the arrows of $Q$ by
\[d_0 > d_1 > \cdots > d_{n-1} > u_0 > u_1 > \cdots > u_{n-1}.\]
The leading terms of the relations under this ordering are $d_id_{i-1}u_{i-1}$ and $d_{i-1}u_{i-1}u_i$ for all $i$. The only overlaps between these terms are of the form $d_id_{i-1}u_{i-1}u_i$. We claim that these overlap ambiguities resolve:
\begin{align*}
(d_id_{i-1}u_{i-1})u_i &-  d_i(d_{i-1}u_{i-1}u_i) \\
	&= (\alpha_i d_iu_id_i + \beta_i u_{i+1}d_{i+1}d_i +\gamma_i d_i) u_i
		- d_i(\alpha_i u_id_iu_i + \beta_i u_iu_{i+1}d_{i+1} +\gamma_i u_i) \\
	&=  \beta_i \left( u_{i+1}(d_{i+1}d_iu_i) -  (d_iu_iu_{i+1})d_{i+1} \right) \\
	&=  \beta_i \left( u_{i+1}(\alpha_{i+1} d_{i+1}u_{i+1}d_{i+1} + \beta_{i+1} u_{i+2}d_{i+2}d_{i+1} +\gamma_{i+1} d_{i+1}) \right. \\
	&\qquad -  \left. (\alpha_{i+1} u_{i+1}d_{i+1}u_{i+1} + \beta_{i+1} u_{i+1}u_{i+2}d_{i+2} +\gamma_{i+1} u_{i+1})d_{i+1} \right) = 0,
\end{align*}
as claimed. Hence, a $\kk$-basis for $\cH$ consists of those paths which avoid the leading terms $d_id_{i-1}u_{i-1}$ and $d_{i-1}u_{i-1}u_i$. It is easily seen that these are the paths \eqref{eq.Aform}.
\end{proof}

Let $\Gamma$ be an abelian monoid. An algebra $A$ is \emph{$\Gamma$-graded} if there is a vector space decomposition $A=\bigoplus_{k \in \Gamma} A_k$ such that $A_k \cdot A_\ell \subset A_{k+\ell}$. A graded algebra $A$ is \emph{locally finite} if $\dim_{\kk} A_k < \infty$ for all $k$.

The quiver down-up algebras $H(\alpha,\beta,0)$ are $\NN$-graded and locally finite via the path-length filtration ($\deg(u_i)=\deg(d_i)=1$ for all $i \in Q_0$). This is the grading we use throughout, however there is also a $\ZZ$-grading on $\cH(\alpha,\beta,\gamma)$ obtained by setting $\deg(u_i)=1$ and $\deg(d_i)=-1$ for all $i \in Q_0$.

Given a quiver $Q$ with $|Q_0|=n$, the associated \emph{adjacency matrix} is $M_Q \in M_n(\NN)$ where $(M_Q)_{ij}$ denotes the number of arrows with source $i$ and target $j$. Let $I$ be a homogeneous ideal in $\kk Q$ and set $A=\kk Q/I$. For each $k \in \NN$, let $H_k \in M_n(\NN)$ where $(H_k)_{ij}=\dim(e_i A e_j)_k$.
Note if $I \subset (\kk Q)^2$, then $H_1 = M_Q$. The \emph{matrix-valued Hilbert series} of $A$ is defined as 
\[ h_A(t) = \sum_{k=0}^\infty H_k t^k.\]
The \emph{total Hilbert series} of $A$, $h_A^{\tot}(t)$, is the infinite series in which the coefficient of $t^k$ is the sum of the entries of $H_k$.

\begin{example}\label{ex.preproj}
Let $Q$ be as in \eqref{eq.dblA}. Define the algebra
\[ A = \kk Q/(d_iu_i - u_{i+1}d_{i+1} : i \in Q_0).\]
The algebra $A$ is a particular preprojective algebra, and it is well-known that $A$ is graded Calabi--Yau algebra of dimension two. Thus, $h_A = (I-Mt+It^2)\inv$ \cite[Theorem 1.2]{RR1}. 

It is easy to see from the relations that a $\kk$-basis for $A$ consists of paths of the form 
\begin{align}\label{eq.preproj}
    u_iu_{i+1} \cdots u_{i+k}d_{i+k}d_{i+k-1} \cdots d_{i+k-\ell}.
\end{align}
For each vertex $i \in Q_0$, there is a bijection between paths of the form \eqref{eq.preproj} with source $e_i$ and the monomial basis of $\kk[a,b]$. It follows that $h_A^{\tot} = n(1-t)^{-2}$.
\end{example}

We now use a similar strategy to establish a basis for the graded quiver down-up algebras.

\begin{corollary}\label{cor.hilb}
Suppose $\gamma=0$. Then the total Hilbert series of $\cH=\cH(\alpha,\beta,0)$ is 
\[ h_{\cH}^{\tot} = n(1-t)^{-2}(1-t^2)\inv.\]
\end{corollary}
\begin{proof}
There is a bijection between paths of the form \eqref{eq.Aform} with source vertex $i$ and monomials $a^{k+1} b^j c^{\ell+1}$ in $\kk[a,b,c]$ where $b$ is given degree two. The Hilbert series of $\kk[a,b,c]$ is $(1-t)^{-2}(1-t^2)\inv$.
\end{proof}

In Lemma \ref{lem.hilb} we compute the matrix-valued Hilbert series of $\cH(\alpha,\beta,0)$.

\section{Generalized Weyl algebras}
\label{sec.gwa}

In this section we establish several key properties of quiver down-up algebras by recognizing them as generalized Weyl algebras. It is well known that the single-vertex down-up algebras appear in this way \cite{KMP}.

Let $R$ be a $\kk$-algebra and $\sigma$ an automorphism of $R$. A \emph{$\sigma$-derivation} of $R$ is a $\kk$-linear map $\delta:R \to R$ satisfying
$\delta(rr') = \sigma(r)\delta(r') + \delta(r)r'$ for all $r,r' \in R$. Given the above, the \emph{Ore extension} $S=R[x;\sigma,\delta]$ is generated over $R$ by $x$ subject to the relations $xr = \sigma(r)x + \delta(r)$ for all $r \in R$. If $R$ is twisted Calabi--Yau of dimension $d$, then $S$ is twisted Calabi--Yau of dimension $d+1$ \cite[Theorem 3.3]{LWW2}.

Suppose that $R=\kk Q/I$ with $I \subset \kk Q_{\geq 2}$ homogeneous. Further, suppose that $\sigma$ and $\delta$ are graded. This implies that $\sigma$ permutes the set of idempotents, so $\sigma(e_i)=e_{\sigma(i)}$, and $\sigma$ commutes with the source and target maps in the sense that, if $a \in Q_1$ with $\fs(a)=e_i$ and $\ft(a)=e_j$, then $\fs(\sigma(a))=e_{\sigma(i)}$ and $\ft(\sigma(a))=e_{\sigma(j)}$. Moreover, $\delta(a) \in e_{\sigma(i)} R_2 e_{j}$. In $S=R[x;\sigma,\delta]$, we have $x = \sum x_i$ where $x_i = xe_i = e_{\sigma(i)}x$. Then for $a \in e_i R e_j$, $x_i a = \sigma(a)x_j + \delta(a)$ is a homogeneous relation from $e_{\sigma(i)}$ to $e_j$. 

As above, let $R$ be an algebra and $\sigma$ an automorphism of $R$. Let $h$ be a regular central element of $R$. The \emph{generalized Weyl algebra} (GWA) $R(\sigma,h)$ is the algebra generated by $R$ and $X^+,X^-$ subject to the relations
\[
	X^-X^+ = h, \qquad X^+X^- = \sigma(h), \qquad
	X^+r = \sigma(r)X^+, \qquad X^-\sigma(r)=rX^-,
\]
for all $r \in R$.

Every GWA is (isomorphic to) a quotient of an iterated Ore extension. Let $S=R[X^-;\sigma][X^+;\sigma\inv,\delta]$ where $\delta(r)=0$ for all $r\in R$ and $\delta(X^-) = \sigma(h)-h$. Then $R(\sigma,h) \iso S/(X^-X^+-h)$, so $R(\sigma,h)$ is noetherian if $R$ is.

Let $R = \bigoplus_{i \in Q_0} \kk[x_i,y_i]$ and denote the identity in $\kk[x_i,y_i]$ by $e_i$ for each $i \in Q_0$. If $\beta_i \neq 0$ for all $i \in Q_0$, then the linear map given by 
\begin{align}\label{eq.gwa_auto}
\sigma(e_i)=e_{i+1}, \quad 
\sigma(x_i)=y_{i+1}, \quad
\sigma(y_i)=\alpha_i y_{i+1} + \beta_i x_{i+1} + \gamma_i,
\end{align} 
defines an automorphism of $R$. Let $x = \sum x_i$ and define the GWA $T=R(\sigma,x)$. We define elements $X_i^-,X_i^+ \in T$ by
$X_i^- = X^-e_i = e_{i+1}X^-$ and $X_i^+ = e_iX^+ = X^+e_{i+1}$. These satisfy the relations
\[ X_i^-r = \sigma\inv(r)X_i^-, \quad X_i^+r = \sigma(r)X_i^+, \quad
X_i^-X_i^+ = x_i, \quad X_i^+X_i^- = \sigma(x_i) = y_{i+1}.\]

\begin{proposition}\label{prop.gwa}
Keep the above setup and assume $\beta_i \neq 0$ for all $i\in Q_0$.
Then $T\iso \cH(\alpha,\beta,\gamma)$.
\end{proposition}
\begin{proof}
Set $\cH = \cH(\alpha,\beta,\gamma)$.
Define a map $\theta:\cH \to T$ by 
$\theta(u_i)=X_i^-$ and $\theta(d_i)=X_i^+$ for all $i\in Q_0$. Then
\begin{align*}
\theta(d_{i-1})\theta(u_{i-1})\theta(u_i) &- \alpha_i \theta(u_i)\theta(d_i)\theta(u_i) - \beta_i \theta(u_i)\theta(u_{i+1})\theta(d_{i+1}) - \gamma_i \theta(u_i) \\
	&= X_{i-1}^+X_{i-1}^-X_i^- - \alpha_i	X_i^-X_i^+X_i^-
		- \beta_i X_i^-X_{i+1}^-X_{i+1}^+ - \gamma_i X_i^- \\
	&= y_iX_i^- - \alpha_i	X_i^-X_i^+X_i^-
		- \beta_i X_i^-X_{i+1}^-X_{i+1}^+ - \gamma_i X_i^- \\
	&= X_i^-\sigma(y_i) - \alpha_i	X_i^-X_i^+X_i^-
		- \beta_i X_i^-X_{i+1}^-X_{i+1}^+ - \gamma_i X_i^- \\
	&= X_i^-(\alpha_i y_{i+1} + \beta_i x_{i+1} + \gamma_i) - \alpha_i X_i^-X_i^+X_i^- - \beta_i X_i^-X_{i+1}^-X_{i+1}^+ - \gamma_i X_i^- = 0, \\
\theta(d_i)\theta(d_{i-1})\theta(u_{i-1}) 
	&- \alpha_i \theta(d_i)\theta(u_i)\theta(d_i) - \beta_i \theta(u_{i+1})\theta(d_{i+1})\theta(d_i) +\gamma_i \theta(d_i) \\
	&= X_i^+X_{i-1}^+X_{i-1}^- - \alpha_i X_i^+X_i^-X_i^+ - \beta_i X_{i+1}^-X_{i+1}^+X_i^+ - \gamma_i X_i^+ \\
	&= X_i^+y_i - \alpha_i X_i^+X_i^-X_i^+ - \beta_i X_{i+1}^-X_{i+1}^+X_i^+ - \gamma_i X_i^+ \\
	&= \sigma(y_i)X_i^+ - \alpha_i X_i^+X_i^-X_i^+ - \beta_i X_{i+1}^-X_{i+1}^+X_i^+ - \gamma_i X_i^+ \\
	&= (\alpha_i y_{i+1} + \beta_i x_{i+1} + \gamma_i)X_i^+ - \alpha_i X_i^+X_i^-X_i^+ - \beta_i X_{i+1}^-X_{i+1}^+X_i^+ - \gamma_i X_i^+ = 0.
\end{align*}
Hence, $\theta$ extends to an algebra homomorphism.

Now define $\theta':T \to \cH$ by 
\[ \theta'(X_i^-) = u_i, \quad \theta'(X_i^+) = d_i, \quad
\theta'(x_i) = X_i^-X_i^+, \quad \theta'(y_i) = X_{i-1}^+X_{i-1}^-.\]
One verifies similarly that $\theta'$ defines an algebra homomorphism, which is clearly inverse to $\theta$.
\end{proof}

\begin{corollary}\label{cor.gwa_props}
Suppose $\beta_i \neq 0$ for all $i\in Q_0$. 
\begin{enumerate}
\item \label{gwa_prop1} The quiver down-up algebra $\cH(\alpha,\beta,\gamma)$ is noetherian.
\item \label{gwa_prop2} The subalgebra $\kk[u_id_i, d_{i-1}u_{i-1}]$ is a polynomial ring in two variables for all $i\in Q_0$.
\end{enumerate}
\end{corollary}
\begin{proof}
Property \eqref{gwa_prop1} is clear. For \eqref{gwa_prop2}, let $\theta$ be as in Proposition \ref{prop.gwa}, so $\theta(u_id_i) = x_i$ and $\theta(d_{i-1}u_{i-1}) = y_i$. As $x_i$ and $y_i$ commute and are algebraically independent in $R$, the same is true of their preimages.
\end{proof}

Let $R$ be a ring with a complete set $e_0,\hdots,e_{n-1}$ of orthogonal idempotents. For example, $R=\kk Q$ with $Q$ a finite quiver. Then $R$ is a \emph{piecewise domain} (PWD) if $ab=0$ implies $a=0$ or $b=0$ for all $a \in e_i R e_k$ and $b \in e_k R e_j$ \cite{GSpwd}. For example, the direct sum of domains is a piecewise domains, and a polynomial extension of a piecewise domain is a piecewise domain. We extend these examples below.

\begin{proposition}\label{prop.pwd}
Let $R$ be a PWD with complete set $e_0,\hdots,e_{n-1}$ of orthogonal idempotents. Let $\sigma$ be an automorphism of $R$.
\begin{enumerate}
\item \label{pwd1} The Ore extension $R[x;\sigma,\delta]$ is a PWD.
\item \label{pwd2} The GWA $R(\sigma,h)$ is a PWD.
\end{enumerate}
\end{proposition}
\begin{proof}
\eqref{pwd1} Set $S=R[x;\sigma,\delta]$. Let $a \in e_i S e_k$ and $b \in e_k S e_j$ be nonzero. By definition of an Ore extension, $a = e_i\sum_{u=0}^d r_u x^u e_k$ and $b = e_k \sum_{v=0}^{d'} r_v x^v e_j$ with $r_u,r_v \in R$ for all $u,v$ and $r_d,r_{d'} \neq 0$. Now the top degree term of $ab$ is 
\[
(e_i r_d x^d e_k)(e_k r_{d'} x^{d'} e_j)
    = (e_i r_d \sigma^d(e_{k})) (\sigma^d(e_{k}) \sigma^d(r_{d'}) \sigma^{d+d'}(e_j)) x^{d+d'} + \text{(lower degree terms)}.
\]
Note that $0 \neq e_i r_d x^d e_k = e_i r_d \sigma^d(e_{k}) x^d$. That is, $e_i r_d \sigma^d(e_{k}) \neq 0$. Similarly, $e_k r_{d'} \sigma^{d'}(e_j) \neq 0$. Thus,
\[ 0 \neq \sigma^d(e_k r_{d'} \sigma^{d'}(e_j))
    = \sigma^d(e_k) \sigma^d(r_{d'}) \sigma^{d+d'}(e_j).\]
Because $R$ is a PWD, then $(e_i r_d \sigma^d(e_{k})) (\sigma^d(e_{k}) \sigma^d(r_{d'}) \sigma^{d+d'}(e_j)) \neq 0$. Thus, $ab \neq 0$, so $S$ is a PWD.

\eqref{pwd2} This is similar and requires only minimal modification. Set $T=R(\sigma,h)$. Let $a\in e_iTe_k$ and $b \in e_kTe_j$. For $k \in \ZZ$, let $X^k = (X^+)^k$ if $k>0$ and $X^k = (X^-)^{-k}$ if $k<0$. Then
$a=e_i\sum_{u=d_1}^{d_2} r_u X^u e_k$ and $b=e_k\sum_{v=d_1'}^{d_2'} r_v X^v e_j$ with $r_u,r_v \in R$ for all $u,v$ and $r_{d_1},r_{d_2},r_{d_1'},r_{d_2'} \neq 0$. The remainder of the proof follows similarly. In particular, the top degree term of $ab$ is $ (e_i r_{d_2} x^{d_2} e_k)(e_k r_{d_2'} x^{d_2'} e_j)$.
\end{proof}

\begin{corollary}\label{cor.pwd}
Suppose $\beta_i \neq 0$ for all $i$. Then $\cH(\alpha,\beta,\gamma)$ is a PWD.
\end{corollary}
\begin{proof}
This follows from Propositions \ref{prop.gwa} and \ref{prop.pwd}.
\end{proof}

\section{Skew group algebras and twisted CY algebras}
\label{sec.skgrp}

In this section we establish that the $\cH(\alpha,\beta,0)$ are twisted graded Calabi--Yau. The \emph{enveloping algebra} of an algebra $A$ is 
defined as $A^e = A \otimes_\kk A^{\text{op}}$. 

\begin{definition}\label{defn.twCY}
An algebra $A$ is \emph{homologically smooth} if $A$ has a finite-length resolution by finitely generated projective $A^e$-modules. The algebra $A$ is \emph{twisted Calabi--Yau} of dimension $d$ if it is homologically smooth and there exists an invertible $(A,A)$-bimodule $U$ such that there are right $A^e$-module isomorphisms 
\[
\Ext_{A^e}^i(A,A^e) \iso \begin{cases}
    0 & i \neq d \\
    U & i=d.
\end{cases}
\]
If $U = {}^1A^\mu$ for an automorphism $\mu$ of $A$, then we call $\mu$ the \emph{Nakayama automorphism}.
\end{definition}

We will use $\mu^A$ to denote the Nakayama automorphism of the algebra $A$, and we will omit the superscript when the algebra is clear. Note that the Nakayama is unique up to composition with an inner automorphism. That is, ${}^1A^\mu \iso {}^1A^\nu$ if and only if $\mu\nu\inv$ is an inner automorphism \cite[Lemma 1.7(d)]{RRZ}.

One can also make appropriate adjustments to the above definition to obtain the definition of twisted \emph{graded} Calabi--Yau algebras. For example, one can insist that the resolution is by \emph{graded} $A^e$-modules. By \cite[Theorem 4.2]{RR2}, a graded algebra $A$ is twisted graded Calabi--Yau if and only if it is twisted Calabi--Yau. In this setting the Nakayama automorphism $\mu^A$ is necessarily graded, and we use the notation $\mu_0^A$ to denote the restriction of $\mu^A$ to $A_0$. A connected graded algebra $A$ is twisted Calabi--Yau if and only if it is Artin--Schelter regular \cite[Lemma 1.2]{RRZ}.

\subsection{Superpotential algebras}

Let $\eta$ be an automorphism of $\kk Q$. An \emph{$\eta$-twisted superpotential} of $\kk Q$ is a linear combination of paths of length $d$ that is invariant under the linear map sending the path $a_1 a_2 \cdots a_d$ to $(-1)^{d+1} \eta(a_d)a_1 \cdots a_{d-1}$. In the name of efficiency, we will often use the following shorthand to write superpotentials. Let $p=a_1 a_2 \cdots a_d$ be a summand of a superpotential $\omega$. Let $S$ denote the set of signed $\eta$-twists of $p$. Then the \emph{compact form} of $p$ is $[p]=\sum_{s \in S} s$.

For each $a \in Q_1$, there is a (linear) map $\delta_a$ sending the path $a_1 a_2 \cdots a_d$ to $a_2 \cdots a_d$ if $a=a_1$ and $0$ otherwise. Given an $\eta$-twisted superpotential $\omega$ and an integer $k \geq 0$, the corresponding \emph{derivation-quotient algebra (of length 1)} is $\kk Q/I$ where $I = (\partial_a \omega : a \in Q_1)$. Hence, $I$ is generated by relations of the same degree $r$. Let $A=\kk Q/I$ be twisted graded Calabi--Yau of dimension three. By \cite[Theorem 6.8, Remark 6.9]{BSW}, $A$ is a derivation-quotient algebra for some $\eta$-twisted superpotential $\omega$.

The following is easy to verify.

\begin{lemma}\label{lem.CY}
Suppose $\beta_i \neq 0$ for all $i\in Q_0$ and let $Q$ be as in \eqref{eq.dblA}.
Define $\eta \in \Aut_{\gr}(\kk Q)$ by
$\eta(u_i)=\beta_{i-1} u_i$ and $\eta(d_i)=\beta_{i-1} d_i$ for all $i\in Q_0$.
Then the algebra $\cH = \cH(\alpha,\beta,0)$ is the derivation-quotient algebra on $\kk Q$ with $\eta$-twisted superpotential
\[ \Omega = \sum_{i\in Q_0} [d_i d_{i-1} u_{i-1} u_i] - \alpha_i [d_i u_i d_i u_i].\]
\end{lemma}

\subsection{Skew group algebras}

Skew group algebras over Artin--Schelter regular algebras provide a plethora of examples of twisted graded Calabi--Yau algebras.

Let $R$ be a ring and $G$ a finite group action on $R$ via automorphism. As a vector space, the \emph{skew group ring} $R\# \kk G$ is $R \otimes \kk G$ and multiplication is defined as 
\[ (r \# g)(s \# h) = r g(s) \# gh,\]
for all $r\# g, s\# h \in R\# \kk G$, extended linearly.

Let $R$ be a twisted graded Calabi--Yau of dimension $d$, generated by $R_1$. Let $G$ be a finite group acting homogeneously on $R$. Then $A=R\# \kk G$ is also twisted graded Calabi--Yau of dimension $d$, generated by $A_0$ and $A_1$ \cite[Theorem 4.1(c)]{RRZ}. Let $e \in A_0$ be a full idempotent such that $B=eAe$ is graded elementary with $B_0=\kk^m$ for some $m$. Then $B$ is Morita equivalent to $A$ and so $B$ is also twisted graded Calabi--Yau of dimension $d$ \cite[Theorem 4.3, Lemma 6.2]{RR1}. Furthermore, $B \iso \kk Q/I$ where $Q$ is the McKay quiver of the action of $G$ on $R_1$ and $I \in \kk Q_{\geq 2}$. The method of computing $Q$ and $I$ is well-known (see e.g., \cite{BSW,RR1}).

Now let $R = \kk\langle u,d \mid d^2u+ud^2, du^2+u^2d \rangle$. The algebra $R$ is a graded down-up algebra, which is Artin--Schelter regular (and Calabi--Yau) of global dimension three. In particular, $R$ is a derivation-quotient algebra with potential $\Omega = [d^2u^2]$ {and $\eta=\id_R$}.  Choose $n \geq 2$ and let $G=\ZZ_n=\langle g \rangle$ act on $R$ by $g(u)=\xi u$ and $g(d)=\xi\inv d$, where $\xi$ is a primitive $n$th root of unity. It is easy to see that the McKay quiver corresponding to the group action of $G$ on $R_1$ is the adjacency matrix $M$ of the quiver $Q$ as in \eqref{eq.dblA}.

In $R\#\kk G$, define the orthogonal idempotents $f_i = \frac{1}{n} \sum_{j=0}^{n-1} \xi^{ij} \# g^j$, so that $f_0 + \cdots + f_{n-1}=1$ is a full idempotent,
and set $B=R\#\kk G$.
Then $B$ is graded Calabi--Yau of dimension three. Since $g(\Omega) = \Omega$, then the \emph{homological determinant} of the $g$-action is trivial \cite[Theorem 3.3]{MS}, and so $\mu^B$ acts trivially on $\kk Q_0$ (see \cite[Section 2.1]{GR1}). 
It follows from \cite[Theorem 5.2]{craw2} that $B$ is a derivation-quotient algebra (with degree four superpotential) and it is not difficult to see from this that $B$ is a quiver down-up algebra. We will provide an explicit proof in this case since it will be useful to trace through some of the details.

\begin{lemma}\label{lem.Bhilb}
Let $B$ be as above. Then
$h_B = (I - Mt + It^2)\inv(I-I t^2)\inv$
and
$h_B^{\tot} = n(1-t)^{-2} (1-t^2)\inv$.
\end{lemma}
\begin{proof}
Since $B=\kk Q/I$ is twisted graded Calabi--Yau of global dimension three,
then by \cite[Theorem 1.3]{RR1},
\[ h_B = I-Mt+Mt^3-Pt^4\]
where $M$ is the adjacency matrix of $Q$ and $P$ is the permutation matrix corresponding to the Nakayama automorphism $\mu^B$. By the above discussion, $Q$ is as in \eqref{eq.dblA} and $P=I$. Thus, $h_B=(p(t))\inv$ where
\[ 
    p(t) = I-Mt+Mt^3-It^4 
        = (1-t^4)I - (t-t^3)M
        = (1-It^2)(I - Mt + It^2).
\]
It suffices to show that the matrix-valued Hilbert series corresponds to the total Hilbert series as stated.
We note that $(I - Mt + It^2)\inv$ is the Hilbert series of the algebra $A$ in Example \eqref{ex.preproj}. By that example, the total Hilbert series of $A$ is $n(1-t^2)\inv$. It follows that $h_B^{\tot} = n(1-t^2)\inv (1-t^2)\inv$.
\end{proof}

\begin{lemma}\label{lem.skgrp}
Keep the above setup. Set $\beta_i=1$ for all $i\in Q_0$ and let $\cH=\cH(0,\beta,0)$. Then $B \iso \cH$. 
\end{lemma}
\begin{proof}
In $B$, set $U_i = f_i (u \# f)$ and $D_i = (d \# f)f_i$ for all $i\in Q_0$. A calculation shows
\[
    U_i = f_i (u\# f) 
        = \left(\frac{1}{n} \sum_{j=0}^{n-1} \xi^{ij} \# g^j\right)(u \# f) 
        = \frac{1}{n} \sum_{j=0}^{n-1} \xi^{ij}(\xi^j u) \# g^j
        = (u\# f)\left(\frac{1}{n} \sum_{j=0}^{n-1} \xi^{(i+1)j} \# g^j\right)
        = (u\#f) f_{i+1}.
\]
Similarly, $D_i = (d \# f)f_i = f_{i+1}(d \# f)$. Now
\begin{align*}
	D_{i-1}U_{i-1}U_i + U_iU_{i+1}D_{i+1}  
        &= (d\#f)f_{i-1} f_{i-1}(u\#f) f_i (u\#f) + f_i(u\# f)f_{i+1}(u\# f)f_{i+2}(d\# f) \\
        &= f_i \left( (d\#f)(u\#f)(u\#f) + (u\# f)(u\# f)(d\# f) \right) \\
        &= f_i (du^2+u^2d) \# f = 0, \\
	D_iD_{i-1}U_{i-1} + U_{i+1}D_{i+1}D_i
        &= (d\# f)f_i (d\#f)f_{i-1} f_{i-1}(u\# f) + f_{i+1}(u\# f) (d\# f)f_{i+1} (d\# f) f_i \\
        &= f_{i+1} \left( (d\# f)(d\# f)(u\# f) + (u\# f)(d\# f)(d\# f) \right) \\
        &= f_{i+1} (du^2+u^2d) \# f = 0.
\end{align*}
Thus, there is a surjective homomorphism $\phi:\cH \to B$ given by $e_i \mapsto f_i$, $u_i \mapsto U_i$, and $d_i \mapsto D_i$ for all $i\in Q_0$. Since $h_{\cH}^{\tot}=h_{B}^{\tot}$ (Corollary \ref{cor.hilb} and Lemma \ref{lem.Bhilb}, respectively), then  $\phi$ is an isomorphism.
\end{proof}

\begin{lemma}\label{lem.hilb}
Let $M$ be the adjacency matrix for $Q$ as in \eqref{eq.dblA}.
Let $\alpha,\beta \in \kk^n$ and set $\cH=\cH(\alpha,\beta,0)$. Then
\[ h_{\cH} = (I-Mt+Mt^3-It^4)\inv = (I-Mt+It^2)\inv (1-I t^2)\inv.\]
\end{lemma}
\begin{proof}
Since $B$ is graded Calabi--Yau of dimension three and $\mu^B$ acts trivially on $\kk Q_0$, 
then by \cite[Theorem 1.3]{RR1} its Hilbert series is $h_{B} = (I-Mt+Mt^3-It^4)\inv$. Thus, by Lemma \ref{lem.skgrp}, the result holds when $\alpha_i=0$ and $\beta_i=1$ for all $i$. However, the $\kk$-basis in Proposition \eqref{prop.basis} does not depend on the choice of $\alpha$ and $\beta$.
\end{proof}

We are now (almost) ready to prove that the $\cH(\alpha,\beta,0)$ are twisted graded Calabi--Yau. We first need one additional definition.

Let $A$ be a locally finite graded algebra. The \emph{graded left socle} of $A$ is $\soc_l(A) = \{ x \in A \mid J(A)x=0\}$ where $J(A)$ is the graded Jacobson radical. Similarly, $\soc_r(A) = \{ x \in A \mid xJ(A)=0\}$ is the \emph{graded right socle} of $A$. If $A=\kk Q/I$ for a homogeneous ideal $I$, then $J(A)$ is the arrow ideal generated by $Q_1$.

\begin{theorem}\label{thm.CY}
Suppose $\beta_i \neq 0$ for all $i \in Q_0$. Then $\cH=\cH(\alpha,\beta,0)$ is twisted-graded Calabi--Yau of global dimension three with Nakayama automorphism $\mu^{\cH}$ given by $\mu^{\cH}(u_i)=-\beta_{i-1}\inv u_i$ and $\mu^{\cH}(d_i)=-\beta_{i-1}\inv d_i$ for all $i \in Q_0$.
\end{theorem}
\begin{proof}
Let $M$ be the adjacency matrix for $Q$ as in \eqref{eq.dblA}. By Lemma \ref{lem.hilb}, $h_{\cH} = (I - Mt + M^T t^3 - It^4)\inv$. Since $\cH$ is a PWD by Corollary \ref{cor.pwd}, then it has trivial graded left and right socle. Since the restriction of $\beta$ to $Q_0$ is trivial, then by \cite[Lemma 8.6]{RR1}, $\cH$ is twisted Calabi--Yau of dimension three. 

By \cite[Theorem 6.8]{BSW}, the Nakayama automorphism of $\cH$ is $\mu^{\cH} = (-1)^{d+1}\eta\inv$. Note that our convention on composing paths is opposite from that reference, hence the inverse.
\end{proof}

It follows from Theorem \ref{thm.CY} that $H(\alpha,\beta,0)$ is Calabi--Yau when $\beta_i=-1$ for all $i$. In this case, let $\Omega$ be the potential from Lemma \ref{lem.CY} and let $\Omega' = \gamma_i[u_id_i]$. Then $\Omega + \Omega'$ is an inhomogeneous potential and it follows from \cite[Theorem 3.6]{BT} that $\cH(\alpha,\beta,\gamma)$ is Calabi--Yau. The next result shows that the twisted Calabi--Yau property holds for all of the quiver down-up algebras with $\beta_i \neq 0$ for all $i$.


\begin{corollary}\label{cor.CY}
Suppose $\beta_i \neq 0$ for all $i \in Q_0$. Then $\cH=H(\alpha,\beta,\gamma)$ is twisted Calabi--Yau.
\end{corollary}
\begin{proof}
Let $F$ be the filtration on $Q$ by path length and give $\cH$ the trivial grading. Then the associated graded ring $\gr_F(\cH)$ is isomorphic to $\cH(\alpha,\beta,0)$, which is twisted Calabi--Yau by Theorem \ref{thm.CY}. By \cite[Theorem 4.4]{ZSL}, $\cH$ is twisted Calabi--Yau.
\end{proof}

\section{Main proof}
\label{sec.main}

We need one additional result before we are ready to prove our main theorem. The following is analogous to \cite[Lemma 4.3]{KMP}. Recall that we interpret subscripts modulo $n$.

\begin{lemma}\label{lem.nnoeth}
Let $\cH=\cH(\alpha,\beta,\gamma)$. 
If $\beta_i =0$ for some $i\in Q_0$, then $\cH$ is not noetherian.
\end{lemma}
\begin{proof}
Let $|Q_0| = n$. Suppose $\beta_i = 0$. We claim that $\cH$ is not right noetherian. Since $\cH = \oplus_{j=0}^{n-1} e_j \cH$, it suffices to show that $e_i A$ is not noetherian as a right $\cH$-module.

We generalize the proof of \cite[Lemma 4.3]{KMP}. For any $m \in Q_0$, set $x_m = d_m u_m$. By Proposition~\ref{prop.basis}, $e_i \cH$ has a $\kk$-basis of the form
\begin{equation}\label{eq.udBasis}
(u_i u_{i+1} \cdots u_{i+k})x_{i+k}^j(d_{i+k}d_{i+k-1}\cdots d_{i+k-\ell})
\end{equation}
with $k,j,\ell \geq 0$.

Since $\beta_i = 0$, the relations in Definition~\ref{defn.qdu} yield
\begin{equation}\label{eq.UDkill}
(\alpha_i u_i d_i +\gamma_i - x_{i-1})u_i 
    = d_i(\alpha_i u_i d_i +\gamma_i - x_{i-1}) = 0
\end{equation}
Note that if $u_k \neq u_i$, then
$(\alpha_i u_i d_i - x_{i-1})u_k = 0$ because of the source and target mismatch. 

Set $U= u_i u_{i+1} \cdots u_{i-1+n}$. That is, $U$ is the product of the distinct $u_m$, starting at $u_i$.  Then for $u_k \neq u_i$,
\[
U(\alpha_i u_i d_i +\gamma_i - x_{i-1})u_k = \gamma_i Uu_k = 0.
\]
Thus
\begin{equation}\label{eq.Ukill}
U(\alpha_i u_i d_i +\gamma_i - x_{i-1})u_m = 0 \text{ for all }m \in \ZZ.
\end{equation}
For each $s \geq 1$, define
\[ I_s = \sum_{m=1}^s U^m(\alpha_i u_i d_i + \gamma_i - x_{i-1})A.\]
By \eqref{eq.Ukill}, 
\[ I_s = \sum_{m=1}^s \sum_{j,k \geq 0} \kk U^m (\alpha_i u_i d_i + \gamma_i - x_{i-1}) x_{i-1}^j d_{i-1} d_{i-2} \cdots d_{i-k}.\]

By \eqref{eq.UDkill}, $d_i x_{i-1} = d_i \alpha_i u_i d_i + \gamma_i d_i = (\alpha_i  x_i  + \gamma_i) d_i$ and hence $d_i x_{i-1}^j = (\alpha_i x_i + \gamma_i)^j d_i$ for any $j \geq 0$.
Thus $I_s$ is contained in the $\kk$-span of the basis monomials
\[
U^m x_{i-1}^{j+1} d_{i-1} d_{i-2} \cdots d_{i-k}, \qquad 
U^m u_i x_i^j d_i d_{i-1} d_{i-2} \cdots d_{i-k} 
\]
with $k, j \geq 0$ and $1 \leq m \leq s$. Thus no polynomial in $I_s$ has $U^{s+1} x_{i-1}$ in its support, while $I_{s+1}$ does contain the polynomial $U^{s+1}(\alpha_i u_i d_i + \gamma_i - x_{i-1})$.

Thus $I_s \subsetneq I_{s+1}$ for all 
{$s \geq 1$} and so $e_i \cH$ is not noetherian. So $\cH$ is not right noetherian. A similar proof shows that $\cH$ is not left noetherian.
\end{proof}

\begin{theorem}\label{thm.main}
Let $\alpha,\beta,\gamma \in \kk^n$ and let $\cH=\cH(\alpha,\beta,\gamma)$. The following are equivalent.
\begin{enumerate}
	\item \label{main.beta} $\beta_i \neq 0$ for all $i \in Q_0$,
	\item \label{main.noeth} $\cH$ is right (or left) noetherian,
	\item \label{main.subalg} $\kk[u_id_i,d_{i-1}u_{i-1}]$ is a polynomial ring in two variables for all $i \in Q_0$,
	\item \label{main.pwd} $\cH$ is a piecewise domain.
\end{enumerate}
\end{theorem}
\begin{proof}
Suppose $\beta_i \neq 0$ for all $i$. Then \eqref{main.noeth} and \eqref{main.subalg} follow from Corollary \ref{cor.gwa_props}, and \eqref{main.pwd} follows from Corollary \ref{cor.pwd}. 

Suppose $\beta_i=0$ for some $i\in Q_0$. By Lemma \ref{lem.nnoeth}, $\cH$ is not noetherian so \eqref{main.noeth} fails. Then we have 
\begin{align}
\label{eq.bad1}    (d_{i-1}u_{i-1}-\alpha_iu_id_i-\gamma_i)u_i &= 0, \\
\label{eq.bad2}    d_i(d_{i-1}u_{i-1} - \alpha_i u_id_i - \gamma_i) &= 0,
\end{align}
which shows that $\cH$ is not a piecewise domain, so \eqref{main.pwd} fails. Moreover, we have $(d_{i-1}u_{i-1}-\alpha_iu_id_i-\gamma_i)u_id_i=0$, so $u_id_i$ and $d_{i-1}u_{i-1}$ are algebraically dependent and \eqref{main.subalg} fails.
\end{proof}

\section{The isomorphism problem}
\label{sec.iso}

In \cite{GZiso}, the authors consider the isomorphism problem for quantum analogues of type A preprojective algebras. The following results are closely related because the algebras live on the same quivers.
The isomorphisms of the next lemma correspond respectively to scaling, rotation, and reflection.

\begin{lemma}\label{lem.iso}
Let $\alpha,\beta,\gamma \in \kk^n$. 

(1) Let $\lambda \in (\kk^\times)^n$. For all $i \in Q_0$, set 
$\alpha_i' = \lambda_i\lambda_{i-1}\inv\alpha_i$,
$\beta_i' = \lambda_{i+1}\lambda_{i-1}\inv\beta_i$, and 
$\gamma_i' = \lambda_{i-1}\inv\gamma_i$.
Then the linear map 
$\phi_\lambda:\cH(\alpha,\beta,\gamma)\to \cH(\alpha',\beta',\gamma')$ 
defined by
\[
    \phi_\lambda(e_i) = e_i, \qquad
    \phi_\lambda(u_i) = \lambda_i u_i, \qquad
    \phi_\lambda(d_i) = d_i,
\]
extends to an isomorphism.

(2) For all $i \in Q_0$, set  
$\alpha_i'=\alpha_{i-1}$, 
$\beta_i'=\beta_{i-1}$, and
$\gamma_i' = \gamma_{i-1}$.
Then linear map $\psi:\cH(\alpha,\beta,\gamma)\to \cH(\alpha',\beta',\gamma')$ defined by
\[
    \psi(e_i) = e_{i+1}, \qquad
    \psi(u_i) = u_{i+1}, \qquad 
    \psi(d_i) = d_{i+1},
\]
extends to an isomorphism.

(3) Suppose $\beta_i\neq 0$ for all $i \in Q_0$.
For all $i \in Q_0$, set 
$\alpha_i'=-\beta_{n-i-1}\inv\alpha_{n-i-1}$, 
$\beta_i'=\beta_{n-i-1}\inv$, and
$\gamma_i' = -\beta_{n-i-1}\inv\gamma_{n-i-1}$.
Then the linear map $\pi:\cH(\alpha,\beta,\gamma)\to \cH(\alpha',\beta',\gamma')$ defined by
\[
    \pi(e_i) = e_{n-i}, \qquad
    \pi(u_i) = d_{n-i-1}, \qquad 
    \pi(d_i) = u_{n-i-1},
\]
extends to an isomorphism.
\end{lemma}
\begin{proof}
The maps are all clearly bijective as linear maps. It suffices to show that each respects the relations on $\cH(\alpha,\beta,\gamma)$. For $\phi_\alpha$:
\begin{align*}
\phi_\lambda(d_{i-1})\phi_\lambda(u_{i-1})\phi_\lambda(u_i) &- \alpha_i \phi_\lambda(u_i)\phi_\lambda(d_i)\phi_\lambda(u_i) - \beta_i \phi_\lambda(u_i)\phi_\lambda(u_{i+1})\phi_\lambda(d_{i+1}) - \gamma_i\phi_\lambda(u_i) \\
	&= \lambda_i(\lambda_{i-1} d_{i-1}u_{i-1}u_i - \lambda_i\alpha_i u_id_iu_i - \lambda_{i+1}\beta_i u_iu_{i+1}d_{i+1} - \gamma_i u_i) \\
	&= \lambda_i\lambda_{i-1} (d_{i-1}u_{i-1}u_i - \alpha_i' u_id_iu_i - \beta_i' u_iu_{i+1}d_{i+1} - \gamma_i' u_i) = 0 \\
\phi_\lambda(d_i)\phi_\lambda(d_{i-1})\phi_\lambda(u_{i-1}) &- \alpha_i \phi_\lambda(d_i)\phi_\lambda(u_i)\phi_\lambda(d_i) - \beta_i \phi_\lambda(u_{i+1})\phi_\lambda(d_{i+1})\phi_\lambda(d_i) - \gamma_i \phi_\lambda(d_i) \\
    &= \lambda_{i-1} d_id_{i-1}u_{i-1} - \lambda_i\alpha_i d_iu_id_i - \lambda_{i+1}\beta_i u_{i+1}d_{i+1}d_i - \gamma_i d_i \\
    &= \lambda_{i-1} (d_id_{i-1}u_{i-1} - \alpha_i' d_iu_id_i - \beta_i' u_{i+1}d_{i+1}d_i - \gamma_i' d_i) = 0.
\end{align*}
For $\psi$:
\begin{align*}
\psi(d_{i-1})\psi(u_{i-1})\psi(u_i) &- \alpha_i \psi(u_i)\psi(d_i)\psi(u_i) - \beta_i \psi(u_i)\psi(u_{i+1})\psi(d_{i+1}) - \gamma_i\psi(u_i) \\
    &= d_iu_iu_{i+1} - \alpha_{i+1}' u_{i+1}d_{i+1}u_{i+1} - \beta_{i+1}' u_{i+1}u_{i+2}d_{i+2} - \gamma_{i+1}' u_{i+1} = 0
     \\
\psi(d_i)\psi(d_{i-1})\psi(u_{i-1}) &- \alpha_i \psi(d_i)\psi(u_i)\psi(d_i) - \beta_i \psi(u_{i+1})\psi(d_{i+1})\psi(d_i) - \gamma_i \psi(d_i) \\
    &= d_{i+1}d_iu_i - \alpha_{i+1}' d_{i+1}u_{i+1}d_{i+1} - \beta_{i+1}' u_{i+2}d_{i+2}d_{i+1} - \gamma_{i+1}' d_{i+1} = 0.
\end{align*}

Finally, for $\pi$:
\begin{align*}
\pi(d_{i-1})&\pi(u_{i-1})\pi(u_i) - \alpha_i \pi(u_i)\pi(d_i)\pi(u_i) - \beta_i \pi(u_i)\pi(u_{i+1})\pi(d_{i+1}) - \gamma_i\pi(u_i) \\
    &= u_{n-i}d_{n-i}d_{n-i-1} - \alpha_i d_{n-i-1}u_{n-i-1}d_{n-i-1} - \beta_i d_{n-i-1}d_{n-i-2}u_{n-i-2} - \gamma_i d_{n-i-1} \\
    &= -\beta_i (d_{n-i-1}d_{n-i-2}u_{n-i-2} - \alpha_{n-i-1}' d_{n-i-1}u_{n-i-1}d_{n-i-1} \\
    &\qquad - \beta_{n-i-1}' u_{n-i}d_{n-i}d_{n-i-1}  - \gamma_{n-i-1}' d_{n-i-1}) = 0
     \\
\pi(d_i)&\pi(d_{i-1})\pi(u_{i-1}) - \alpha_i \pi(d_i)\pi(u_i)\pi(d_i) - \beta_i \pi(u_{i+1})\pi(d_{i+1})\pi(d_i) - \gamma_i \pi(d_i) \\
    &= u_{n-i-1}u_{n-i}d_{n-i} - \alpha_i u_{n-i-1}d_{n-i-1}u_{n-i-1} - \beta_i d_{n-i-2}u_{n-i-2}u_{n-i-1} - \gamma_i u_{n-i-1} \\
    &= - \beta_i (d_{n-i-2}u_{n-i-2}u_{n-i-1} - \alpha_{n-i-1}' u_{n-i-1}d_{n-i-1}u_{n-i-1} \\
    &\qquad - \beta_{n-i-1}' u_{n-i-1}u_{n-i}d_{n-i}  - \gamma_{n-i-1}' u_{n-i-1}) = 0.
\end{align*}
Hence, each linear map extends to an isomorphism.
\end{proof}

In general, we do not know the answer to the isomorphism problem for the quiver down-up algebras. However, in the graded case with $n\geq 3$ vertices, we obtain a complete classification of isomorphism classes.

\begin{theorem}\label{thm.iso}
Fix $n \geq 3$. Let $\alpha,\alpha' \in \kk^n$ and $\beta,\beta' \in (\kk^\times)^n$. Then $\cH(\alpha,\beta,0) \iso \cH(\alpha',\beta',0)$ if and only if there exists $\lambda\in(\kk^\times)^n$ and $k \in Q_0$ such that one of the following hold:
\begin{enumerate}
	\item \label{iso1} $\alpha_i' = \lambda_{i+k}\lambda_{i+k-1}\inv \alpha_{i+k}$,\quad
	$\beta_i' = \lambda_{i+k+1} \lambda_{i+k-1}\inv \beta_{i+k}$, or

	\item \label{iso2} 
	$\alpha_i' = -\lambda_{n-i-k}\inv\lambda_{n-i-k-1} \beta_{n-i-k-1}\inv\alpha_{n-i-k-1}$, \quad
	$\beta_i' = \lambda_{n-i-k}\inv\lambda_{n-i-k-2}\inv \beta_{n-i-k-1}\inv$.
\end{enumerate}
\end{theorem}
\begin{proof}
If \eqref{iso1} holds for some $\lambda_i \in \kk^\times$, then by Lemma \ref{lem.iso}, $\psi^{-k} \circ \phi_\lambda$ gives the intended isomorphism. If \eqref{iso2} holds for some $\lambda_i \in \kk^\times$, then by Lemma \ref{lem.iso}, $\psi^{-k} \circ \pi \circ \phi_\lambda$ gives the intended isomorphism. 

Set $\cH=\cH(\alpha,\beta,0)$ and $\cH'=\cH(\alpha',\beta',0)$.
Assume $\Phi:\cH \to \cH'$ is an isomorphism. 
By \cite[Theorem 8]{Gpath}, we may assume that $\Phi$ is a graded isomorphism.
Then $\Phi$ restricts to a permutation of the set of vertices $\{e_i\}$. Since $u_i$ is an arrow from $e_i$ to $e_{i+1}$, then this permutation preserves adjacency. Hence, $\Phi$ may be regarded as an element of the dihedral group on $n$ vertices.

First, suppose that $\Phi(e_i)=e_i$ for all $i$. As we have assumed $n \geq 3$, then $Q$, as in \eqref{eq.dblA}, is schurian. 
That is, for $i \neq j$, there is at most one arrow $a$ with $s(a)=i$ and $t(a)=j$.
Thus, there are scalars $\mu_i,\nu_i \in \kk^\times$ such that $\Phi(u_i)=\mu_i u_i$ and $\Phi(d_i)=\nu_i d_i$ for all $i \in Q_0$. Then
\begin{align*}
0 &= \Phi(d_{i-1}u_{i-1}u_i - \alpha_i u_id_iu_i - \beta_i u_iu_{i+1}d_{i+1}) \\
	&= \mu_i (\mu_{i-1}\nu_{i-1} d_{i-1}u_{i-1}u_i - \alpha_i \mu_i\nu_i u_id_iu_i - \beta_i \mu_{i+1}\nu_{i+1} u_iu_{i+1}d_{i+1}) \\	
	&= \mu_i\mu_{i-1}\nu_{i-1} (d_{i-1}u_{i-1}u_i - \alpha_i (\mu_i\nu_i)(\mu_{i-1}\nu_{i-1})\inv u_id_iu_i - \beta_i (\mu_{i+1}\nu_{i+1})(\mu_{i-1}\nu_{i-1})\inv u_iu_{i+1}d_{i+1}) \\	
0 &= \Phi(d_id_{i-1}u_{i-1} - \alpha_i d_iu_id_i - \beta_i u_{i+1}d_{i+1}d_i ) \\
	&= \nu_i (\mu_{i-1}\nu_{i-1} d_id_{i-1}u_{i-1} - \mu_i\nu_i \alpha_i d_iu_id_i - \beta_i \mu_{i+1}\nu_{i+1} u_{i+1}d_{i+1}d_i) \\
	&= \nu_i (d_id_{i-1}u_{i-1} - (\mu_i\nu_i)(\mu_{i-1}\nu_{i-1})\inv \alpha_i d_iu_id_i - \beta_i (\mu_{i+1}\nu_{i+1})(\mu_{i-1}\nu_{i-1})\inv u_{i+1}d_{i+1}d_i).
\end{align*}
Setting $\lambda_i = \mu_i\nu_i$ gives (1) for $k=0$.

Now assume that $\Phi$ corresponds to a (nontrivial) rotation, so $\Phi(e_i)=e_{i-k}$ for some $k$, $0 < k < n$. Set $\alpha_i'' = \alpha_{i-k}'$ and $\beta_i'' = \beta_{i-k}'$ and let $\cH''=\cH(\alpha'',\beta'',0)$.
Thus, $\psi^k:\cH' \to \cH''$ is an isomorphism. Then $\psi^k \circ \Phi:\cH \to \cH''$ is an isomorphism that fixes the vertices. Consequently, there are scalars $\lambda_i \in \kk^\times$ such that $\alpha_i'' = \lambda_i\lambda_{i-1}\inv \alpha_i$ and $\beta_i'' = \lambda_{i+1}\lambda_{i-1}\inv \beta_i$. Thus,
\begin{align*}
    \alpha_{i-k}' = \lambda_i\lambda_{i-1}\inv \alpha_i 
        &\Rightarrow \alpha_i' = \lambda_{i+k}\lambda_{i+k-1}\inv \alpha_{i+k} \\
    \beta_{i-k}' = \lambda_{i+1}\lambda_{i-1}\inv \beta_i
        &\Rightarrow \beta_i' = \lambda_{i+k+1}\lambda_{i+k-1}\inv \beta_{i+k}.
\end{align*}

Finally, assume $\Phi$ corresponds to a reflection, so $\Phi(e_i)=e_{n-i-k}$ for some $k \in Q_0$. Define $\cH'' = \cH(\alpha'',\beta'',0)$ as above with $\alpha_i'' = \alpha_{i-k}'$ and $\beta_i'' = \beta_{i-k}'$, so that $\psi^k:\cH' \to \cH''$ is an isomorphism. Set $\alpha_i'''=-(\beta_{n-i-1}'')\inv(\alpha_{n-i-1}'')$ and $\beta_i'''=(\beta_{n-i-1}'')\inv$ and let $\cH'''=\cH(\alpha''',\beta''',0)$ so that $\pi:\cH'' \to \cH'''$ is an isomorphism.
Then $\pi \circ \psi^k \circ \Phi$ is an isomorphism that fixes the vertices, so there exists scalars $\lambda_i \in \kk^\times$ such that $\alpha_i''' = \lambda_i\lambda_{i-1}\inv \alpha_i$ and $\beta_i''' = \lambda_{i+1}\lambda_{i-1}\inv \beta_i$. Unwinding, we have
\begin{align*}
    \beta_i' &= \beta_{i+k}'' 
        = (\beta_{n-i-k-1}''')\inv
        = \lambda_{n-i-k}\inv \lambda_{n-i-k-2}\beta_{n-i-k-1}\inv \\
    \alpha_i' &= \alpha_{i+k}''
        = -(\beta_{n-i-k-1}''')\inv(\alpha_{n-i-k-1}''') \\
        &= -\left( \lambda_{n-i-k}\inv \lambda_{n-i-k-2} \beta_{n-i-k-1}\inv\right)\left( \lambda_{n-i-k-1}\lambda_{n-i-k-2} \alpha_{n-i-k-1}\right) \\
        &= -\lambda_{n-i-k}\inv \lambda_{n-i-k-1}\beta_{n-i-k-1}\inv\alpha_{n-i-k-1}.
\end{align*}
This gives the result.
\end{proof}

The isomorphism problem for $n=1$ is known (see \cite[Proposition 3.14]{Ghberg} translated to down-up algebras). We anticipate that Theorem \ref{thm.iso} holds in the case $n=2$. However, because the quiver in this case is not schurian, then more work is required. See \cite[Proposition 3.12]{GZiso}. 


\begin{thebibliography}{10}

\bibitem{BRdu}
G.~Benkart and T.~Roby.
\newblock Down-up algebras.
\newblock {\em J. Algebra}, 209(1):305--344, 1998.

\bibitem{BRduADD}
G.~Benkart and T.~Roby.
\newblock Addendum: ``{D}own-up algebras''.
\newblock {\em J. Algebra}, 213(1):378, 1999.

\bibitem{BWhopf}
G.~Benkart and S.~Witherspoon.
\newblock A {H}opf structure for down-up algebras.
\newblock {\em Math. Z.}, 238(3):523--553, 2001.

\bibitem{BT}
R.~Berger and R.~Taillefer.
\newblock Poincar\'e-{B}irkhoff-{W}itt deformations of {C}alabi-{Y}au algebras.
\newblock {\em J. Noncommut. Geom.}, 1(2):241--270, 2007.

\bibitem{BSW}
R.~Bocklandt, T.~Schedler, and M.~Wemyss.
\newblock Superpotentials and higher order derivations.
\newblock {\em J. Pure Appl. Algebra}, 214(9):1501--1522, 2010.

\bibitem{craw2}
S.~Crawford.
\newblock Superpotentials and quiver algebras for semisimple {H}opf actions.
\newblock {\em Algebr. Represent. Theory}, 26(6):2709--2752, 2023.

\bibitem{Ghberg}
J.~Gaddis.
\newblock Two-parameter analogs of the {H}eisenberg enveloping algebra.
\newblock {\em Comm. Algebra}, 44(11):4637--4653, 2016.

\bibitem{Gpath}
J.~Gaddis.
\newblock Isomorphisms of graded path algebras.
\newblock {\em Proc. Amer. Math. Soc.}, 149(4):1395--1403, 2021.

\bibitem{GK2}
J.~Gaddis and D.~Keeler.
\newblock Normal extensions of dimension two twisted graded {C}alabi--{Y}au
  algebras.
\newblock {\em In preparation}, 2026.

\bibitem{GR1}
J.~Gaddis and D.~Rogalski.
\newblock Quivers supporting twisted {C}alabi-{Y}au algebras.
\newblock {\em J. Pure Appl. Algebra}, 225(9):Paper No. 106645, 33, 2021.

\bibitem{GZiso}
J.~Gaddis and D.~Zazycki.
\newblock Dimension two twisted graded {C}alabi--{Y}au algebras on two-vertex
  quivers.
\newblock {\em arXiv preprint:2508.01950}, 2025.

\bibitem{GSpwd}
R.~Gordon and L.~W. Small.
\newblock Piecewise domains.
\newblock {\em J. Algebra}, 23:553--564, 1972.

\bibitem{KKZ6}
E.~Kirkman, J.~Kuzmanovich, and J.~J. Zhang.
\newblock Invariant theory of finite group actions on down-up algebras.
\newblock {\em Transform. Groups}, 20(1):113--165, 2015.

\bibitem{KMP}
E.~Kirkman, I.~M. Musson, and D.~S. Passman.
\newblock Noetherian down-up algebras.
\newblock {\em Proc. Amer. Math. Soc.}, 127(11):3161--3167, 1999.

\bibitem{kulk}
R.~S. Kulkarni.
\newblock Down-up algebras and their representations.
\newblock {\em J. Algebra}, 245(2):431--462, 2001.

\bibitem{LWW2}
L.~Liu, S.~Wang, and Q.~Wu.
\newblock Twisted {C}alabi-{Y}au property of {O}re extensions.
\newblock {\em J. Noncommut. Geom.}, 8(2):587--609, 2014.

\bibitem{MS}
I.~Mori and S.~P. Smith.
\newblock {$m$}-{K}oszul {A}rtin-{S}chelter regular algebras.
\newblock {\em J. Algebra}, 446:373--399, 2016.

\bibitem{RRZ}
M.~Reyes, D.~Rogalski, and J.~J. Zhang.
\newblock Skew {C}alabi-{Y}au algebras and homological identities.
\newblock {\em Adv. Math.}, 264:308--354, 2014.

\bibitem{RR1}
M.~L. Reyes and D.~Rogalski.
\newblock Growth of graded twisted {C}alabi-{Y}au algebras.
\newblock {\em J. Algebra}, 539:201--259, 2019.

\bibitem{RR2}
M.~L. Reyes and D.~Rogalski.
\newblock Graded twisted {C}alabi-{Y}au algebras are generalized
  {A}rtin-{S}chelter regular.
\newblock {\em Nagoya Math. J.}, 245:100--153, 2022.

\bibitem{ZSL}
G.-S. Zhou, Y.~Shen, and D.-M. Lu.
\newblock Skew {C}alabi-{Y}au property of normal extensions.
\newblock {\em Manuscripta Math.}, 161(1-2):125--140, 2020.

\end{thebibliography}

\end{document}